\documentclass[12pt]{article}
\usepackage{amsmath,amssymb,amsthm}
\usepackage{enumerate}
\usepackage{graphicx}
\usepackage[T1]{fontenc}
\usepackage{caption}
\usepackage{cite}

\def\draw #1 by #2 (#3){
	\vbox to #2{
		\hrule width #1 height 0pt depth 0pt
		\vfill
		\special{picture #3} 
	}
}

\def\scaleddraw #1 by #2 (#3 scaled #4){{
		\dimen0=#1 \dimen1=#2
		\divide\dimen0 by 1000 \multiply\dimen0 by #4
		\divide\dimen1 by 1000 \multiply\dimen1 by #4
		\draw \dimen0 by \dimen1 (#3 scaled #4)}
}

\newtheorem{theorem}{Theorem}[section]
\newtheorem{example}[theorem]{Example}
\newtheorem{problem}[theorem]{Problem}
\newtheorem{defin}[theorem]{Definition}
\newtheorem{lemma}[theorem]{Lemma}

\newtheorem{remark}[theorem]{Remark}
\usepackage{xcolor}
\usepackage{float}
\usepackage{graphicx,subcaption,lipsum}
\usepackage{mathtools}

\newtheorem{nt}{Note}

\newcommand{\singlespacing}{\let\CS=\@currsize\renewcommand{\baselinestretch}{1}\tiny\CS}
\newcommand{\oneandahalfspacing}{\let\CS=\@currsize\renewcommand{\baselinestretch}{1.25}\tiny\CS}
\newcommand{\doublespacing}{\let\CS=\@currsize\renewcommand{\baselinestretch}{1.35}\tiny\CS}

\newtheorem{rule-def}[theorem]{Rule}

\begin{document}
\numberwithin{equation}{section}.
\numberwithin{table}{section}
\numberwithin{figure}{section}
	\baselineskip 16pt
	\newcommand{\la}{\lambda}
	\newcommand{\si}{\sigma}
	\newcommand{\ol}{1-\lambda}
	\newcommand{\be}{\begin{equation}}
		\newcommand{\ee}{\end{equation}}
	\newcommand{\bea}{\begin{eqnarray}}
		\newcommand{\eea}{\end{eqnarray}}

	\baselineskip=0.30in
		\baselineskip=0.30in
         \textbf{\Large{On extremal graphs with respect to the $\mathcal{ABS}$ index}}
	\vspace*{0.3cm} 
\begin{center}
Swathi Shetty$^{1}$, B. R. Rakshith$^{*,2}$, Sayinath Udupa N. V. $^{3}$\\
Department of Mathematics, Manipal Institute of Technology\\ Manipal Academy of Higher Education\\ Manipal, India -- 576104\\
swathi.dscmpl2022@learner.manipal.edu$^{1}$\\
ranmsc08@yahoo.co.in$^{*,2}$\\
sayinath.udupa@manipal.edu	$^{3}$.
\end{center}
\[\text{July 23, 2025}^\dagger\]
\footnotetext{*Corresponding author; $\dagger$ Initial journal submission date}
\begin{abstract}       Recently, Ali et al. \cite{ali2024extremal} posed several open problems concerning extremal graphs with respect to the $\mathcal{ABS}$ index. These problems involve characterizing graphs that attain the maximum $\mathcal{ABS}$ index within specific graph classes, including: (i) connected graphs with \( n \) vertices and \( p \) cut-vertices; (ii) connected graphs of order \( n \) with vertex $k$-partiteness \( v_k(G) = r \); and (iii) connected bipartite graphs of order \( n \) with a fixed vertex connectivity \( \kappa \). In this paper, we provide complete solutions to all of these problems.   \end{abstract}
        \section{Introduction}
        Throughout the paper, $G$ stands for a graph on $n$ vertices without any loops and multiple-edges. As usual, we denote the vertex set and the edge set of $G$ by $V(G)$ and $E(G)$, respectively. Two adjacent vertices $v_{i}$ and $v_{j}$ are represented as $v_{i}\sim v_{j}$, or $i\sim j$. The degree of a vertex $v_i$ in $G$ is denoted as $d_G(v_i)/ d(v_i).$ The vertex connectivity $\kappa$ of $G$ is the cardinality of a minimum vertex cut set. The vertex $k$-partiteness  of $G$, $v_{k}(G)$ is the minimum number of vertices whose removal results in a $k$-partite graph.   Topological indices, a well-known class of graph invariants, are numerical values derived from the structure of a graph. In molecular graph theory, these indices are widely used as they often reflect the physicochemical properties of the corresponding molecules. For some recent studies on topological indices, see \cite{das2025vertex,gao2025extremal}.  The atom-bond sum connectivity index, commonly known as the $\mathcal{ABS}$ index, is a recently introduced topological invariant that has quickly attracted significant attention. It is defined as  \begin{center}$\mathcal{ABS}(G)=\sum\limits_{uv\in E(G)} \sqrt{1-\dfrac{2}{d(u)+d(v)}}$.\end{center} The chemical relevance of the $\mathcal{ABS}$ index is explored in \cite{ali2023atom}, while its extremal values and bounds are investigated in \cite{li2024greatest,nithya2023smallest}. 

Recently, in the review paper \cite{ali2024extremal}, Ali et al. listed known bounds and extremal results related to the $\mathcal{ABS}$ index. They also stated several new extremal results that follow easily from existing general findings. Furthermore, the authors posed several open problems concerning extremal graphs with respect to the $\mathcal{ABS}$ index. These problems involve characterizing graphs that attain the maximum $\mathcal{ABS}$ index within specific graph classes, including: (i) connected graphs with \( n \) vertices and \( p \) cut-vertices; (ii) connected graphs of order \( n \) with vertex $k$-partiteness \( v_k(G) = r \); and (iii) connected bipartite graphs of order \( n \) with a fixed vertex connectivity \( \kappa \). Motivated by these works, in this paper, we provide complete solutions to all of these problems.
      \section{$\mathcal{ABS}$ index of graphs with the given number of cut-vertices}\label{given1}
Let $\mathbb{V}_{n}^p$ be the collection of all  connected graphs of order $n$ with $p$ number of cut-vertices. We need the following theorem to prove our results in this section.
\begin{theorem}~\rm\cite{wu2017monotonicity,ali2024extremal}\label{clique theorem}
If $G$ is a graph attaining the maximum $\mathcal{ABS}$ index in $\mathbb{V}_{n}^p$, then every block of $G$ is a clique and every cut-vertex of $G$ is contained in exactly two blocks.
\end{theorem}
\begin{lemma}\label{lemma clique}
Let $G$ be a graph that attains the maximum $\mathcal{ABS}$ index in $\mathbb{V}_{n}^p$. Then $G$ is either the path graph $P_{n}$ or it can be formed by attaching one of the end vertices of a path $P^{i}$ (with a length of at least 1) to a vertex $u_i$ of the complete graph $K_{n-p}$ (with vertex labels $u_1,u_2,\ldots,u_{n-p}$) for $1\le i \le k$, where $1\le k\le n-p$.
\end{lemma}
\begin{proof}
By Theorem~\ref{clique theorem}, every block of $G$ is a clique and every cut-vertex of $G$ is contained in exactly two blocks.
Let $H_{1}=K_{\alpha}$ and $H_{2}=K_{\beta}$ be the two largest cliques in $G$, where $\alpha \ge \beta \ge 2$. Assume that $\beta\ge 3$.\\[2mm]
Case 1: Suppose $H_{1}$ and $H_{2}$ share a common cut-vertex, say $x$ of $G$.   
Let $y$ be a vertex of minimum degree in $H_{2}\backslash \{x\}$, and $z$ be a vertex of $H_{2}$ that is different from the vertices $x$  and $y$. Let $G^{\prime}$ be the graph obtained from $G$ by removing all the edges of $H_{2}$ which are incident with $y$ except the edge $xy$, and then joining each vertex of $H_{2}\backslash \{x,y\}$ with every vertex in $H_{1}\backslash \{x\}$. Then $d_{G^{\prime}}(u)= d_{G}(u)+\beta-2$ for $u\in V(H_{1})\backslash\{x\}$, and $d_{G^{\prime}}(w)=d_{G}(w)+\alpha-2$ for $w\in V(H_{2})\backslash\{x,y\}$. Also, $d_{G}(x)=d_{G^{\prime}}(x)=\alpha+\beta-2$. Importantly, the cut-vertices of $G$ and $G^{\prime}$ are essentially the same. We now consider the following subcases.\\[2mm]
Subcase 1.1: The vertex $y$ is not a cut-vertex of $G$. Then $d_{G}(y)=\beta-1$ and $d_{G^{\prime}}(y)=1$. Therefore,    
$\mathcal{ABS}(G)-\mathcal{ABS}(G^{\prime})<$\\[2mm]$\displaystyle\sum_{w: w\in V(H_{2})\backslash \{y\}}\sqrt{1-\dfrac{2}{d_{G}(w)+d_{G}(y)}}-\sum_{u:u\in V(H_{1})\backslash\{x\}}\,\,\sum_{w: w\in V(H_{2})\backslash\{x,y\}}\sqrt{1-\dfrac{2}{d_{G^{\prime}}(u)+d_{G^{\prime}}(w)}}$.\\[2mm]
Since $d_{G}^{\prime}(u)+d_{G^{\prime}}(w)> d_{G}(u)+d_{G}(y)$  and $|V(H_{1})|\ge 3$, $\mathcal{ABS}(G)<\mathcal{ABS}(G^{\prime})$, a contradiction.\\[2mm]
Subcase 1.2: The vertex $y$ is a cut-vertex of $G$. Then $d_{G}(y)\le 2\beta-2$ and $d_{G^{\prime}}(y)\le\beta$. Let $u_{1}$ ($\neq x$) be a vertex in $H_{1}$. If |$V(H_{2})|\ge 4$, then $V(H_{2})\backslash\{x,y,z\}\neq\emptyset$. Let $w_{1}$ be a vertex in $V(H_{2})\backslash\{x,y,z\}$. Then $\mathcal{ABS}(G)-\mathcal{ABS}(G^{\prime})<$
\begin{align*}&\displaystyle\sum_{w: w\in V(H_{2})\backslash \{y\}}\sqrt{1-\dfrac{2}{d_{G}(w)+d_{G}(y)}}
-\sum_{w: w\in V(H_{2})\backslash\{x,y\}}\sqrt{1-\dfrac{2}{d_{G^{\prime}}(w)+d_{G^{\prime}}(u_{1})}}\\[2mm]&+\sum_{\underset{ y^{\prime}\notin  V(H_{2})}{y^{\prime}: y^{\prime}\sim y} }\sqrt{1-\dfrac{2}{d_{G}(y^{\prime})+d_{G}(y)}}-\sum_{u: u\in V(H_{1})\backslash\{x,u_{1}\}}\sqrt{1-\dfrac{2}{d_{G^{\prime}}(u)+d_{G^{\prime}}(z)}}\\[2mm]&-\sum_{u: u\in V(H_{1})\backslash\{x,u_{1}\}}\sqrt{1-\dfrac{2}{d_{G^{\prime}}(u)+d_{G^{\prime}}(w_{1})}}<0.\end{align*}
Thus, $\mathcal{ABS}(G)<\mathcal{ABS}(G^{\prime})$, a contradiction. Otherwise $|V(H_{2})|=3$. In this case,
$\mathcal{ABS}(G)-\mathcal{ABS}(G^{\prime})<$\\
\begin{align}\label{cuteq1}\allowdisplaybreaks&\displaystyle\sum_{w:w\in\{x,z\}}\sqrt{1-\dfrac{2}{d_{G}(w)+d_{G}(y)}}-\sum_{u:u\in V(H_{1})\backslash\{x\}}\sqrt{1-\dfrac{2}{d_{G^{\prime}}(u)+d_{G^{\prime}}(z)}}\notag\\[2mm]&+
\sum_{\underset{ y^{\prime}\notin  V(H_{2})}{y^{\prime}: y^{\prime}\sim y} }\sqrt{1-\dfrac{2}{d_{G}(y^{\prime})+d_{G}(y)}}-\sum_{\underset{ y^{\prime}\notin  V(H_{2})}{y^{\prime}: y^{\prime}\sim y} }\sqrt{1-\dfrac{2}{d_{G^{\prime}}(y^{\prime})+d_{G^{\prime}}(y)}}\notag\\[2mm]
&-\sqrt{1-\dfrac{2}{d_{G^{\prime}}(y^{\prime})+d_{G^{\prime}}(x)}}.
\end{align}
The sum of the first two terms in (\ref{cuteq1}) is not positive, and it is easy to verify that the sum of the last three terms is negative. Thus, $\mathcal{ABS}(G)<\mathcal{ABS}(G^{\prime})$, again a contradiction.\\[2mm] 
Case 2: Suppose $H_{1}$ and $H_{2}$ do not share a common cut-vertex. Let $x$ be a cut-vertex of $G$ in $H_{2}$, and let $y$ be a vertex of minimum degree in $H_{2}\backslash \{x\}$. Let $G^{\prime}$ be the graph obtained from $G$ by removing all the edges of $H_{2}$ that are incident with either $x$ or $y$, except the edge $xy$, and then joining each vertex of $H_{2}\backslash \{x,y\}$ with every vertex in $H_{1}$. Then the cut-vertices of $G$ and $G^{\prime}$ are essentially the same, and by a similar argument to that of Case 1, we end up with a contradiction.\\[2mm] 
Thus, $\beta=2$. Since every block in $G$ is a clique, $G$ is either a path graph $P_{n}$ or it can be formed by attaching one of the end vertices of a path $P^{i}$ (with a length of at least 1) to a vertex $u_i$ of the complete graph $K_{n-p}$ (with vertex labels $u_1,u_2,\ldots,u_{n-p}$) for $1\le i \le k$, where $1\le k\le n-p$.
\end{proof}
\noindent
Graphs $\mathbb{K}_{n}^{p}$: Let $n\ge 3$ and $0\le p\le n-2$. If $n\ge 2p$, then  $\mathbb{K}_n^p$ is obtained by attaching $p$ pendant vertices to $p$ vertices of the complete graph $K_{n-p}$. Otherwise ($n<2p$), the graph  $\mathbb{K}_n^p$ is formed by attaching $n-p-1$ pendant vertices to $n-p-1$ vertices of $K_{n-p}$, and then a path of length $2p-n+1$ is attached to the remaining one vertex of $K_{n-p}$.
\begin{figure}[H]
\includegraphics[width=1\textwidth]{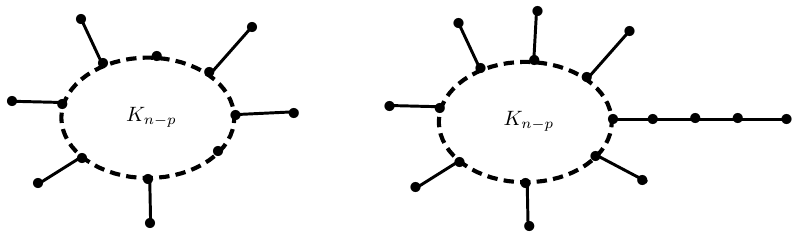}
\label{fig11}
\caption{Graphs $\mathbb{K}_{n}^{p}$, for $n\ge 2p$ (left) and $n<2p$ (right).}
\end{figure}
 \begin{theorem}
 Let $G\in \mathbb{V}_n^p.$ Then $\mathcal{ABS}(G)\le\mathcal{ABS}(\mathbb{K}_{n}^{p})$, 
where the equality holds if and only if $G\cong \mathbb{K}_n^p.$
 \end{theorem}
 \begin{proof}
Suppose $G$ has maximum $\mathcal{ABS}$ index. If $G\cong P_{n}$, then there is nothing to prove. Otherwise, $G\ncong P_{n}$, and by Lemma \ref{lemma clique}, $G$ can be constructed by attaching one of the end vertices of a path $P^{i}$ (with a length of at least 1) to a vertex $u_i$ of the complete graph $K_{n-p}$ (with vertex labels $u_1,u_2,\ldots,u_{n-p}$) for $1\le i \le k$, where $1\le k\le n-p$ and $1\le p\le n-3$. Assume that the length of $P^{1}$ is at least  2 and  $k< n-p$. Consider $G^{\prime}$, the graph obtained from $G$ by deleting the pendant vertex of the path $P^{1}$ ($w_0(u_{1})-w_{1}-w_{2}-\cdots-w_{\ell},\, \ell\ge 2$), and then adding a pendant edge to the vertex $u_{n-p}$.\\[2mm] 
Now,
$\mathcal{ABS}(G)-\mathcal{ABS}(G^{\prime})= \dfrac{1}{\sqrt{3}}+\sqrt{1-\dfrac{2}{d_{G}(w_{\ell-2})+2}}-\sqrt{1-\dfrac{2}{d_{G^{\prime}}(w_{\ell-2})+1}}\\[2mm]-\sqrt{1-\dfrac{2}{n-p+1}}+(n-p-1)\left(\sqrt{1-\dfrac{2}{2(n-p-1)}}-\sqrt{1-\dfrac{2}{2(n-p)-1}}\right)\\[2mm]< \dfrac{1}{\sqrt{2}}-\sqrt{1-\dfrac{2}{n-p+1}}\le 0.$\\[2mm]
Thus, $ABS(G)<ABS(G^{\prime})$, a contradiction. Therefore, length of $P^{i}$ is 1 for $1\le i\le k<n-p$, or $k=n-p$.\\[2mm] 
 If the length of $P^{i}$ is 1 for $1\le i\le k<n-p$, then we are done. Otherwise, $k=n-p$. Without loss of generality, we can assume that $P^{1}$ is of maximum length among all paths $P^{i}$. Let $G^{\prime\prime}$ be the graph obtained from $G$ by removing the pendant vertex in $P^{i}$ ($i\neq1$), and then adding a pendant edge to the pendant vertex in $P^{1}$. Then  
$\mathcal{ABS}(G)=\mathcal{ABS}(G^{\prime\prime})$ when the length of paths $P^{1}$ and $P^{i}$ are at least 3, and if the length of $P^{i}$ is 2 and the length of  $P^{1}$ is at least $2$, then   
$\mathcal{ABS}(G)-\mathcal{ABS}(G^{\prime\prime})=\sqrt{1-\dfrac{2}{n-p+2}}-\sqrt{1-\dfrac{2}{n-p+1}}+\dfrac{1}{\sqrt{3}}-\dfrac{1}{\sqrt{2}}$$\le \sqrt{\dfrac{3}{5}}-\sqrt{2}+\dfrac{1}{\sqrt{3}}<0.$ That is, $\mathcal{ABS}(G)<\mathcal{ABS}(G^{\prime\prime})$. Hence the length of $P^{1}$ is at least 1 and the lengths of $P^{i}$ is 1 for $2\le i\le n-p$. Proving that $G\cong \mathbb{K}_{n}^{p}$.
 \end{proof}
 \section{$\mathcal{ABS}$ index and vertex $k$-partiteness in graphs}\label{given2}
 A complete multipartite graph with $k$-partition sets of order $n_{1}, n_{2},\ldots, n_{k}$ is denoted as $K_{n_1,n_2,\ldots,n_k}$. 
 The Tur$\acute{\text{a}}$n graph $T(n,k)$ is the complete $k$-partite graph on $n$ vertices given by $T(n,k)\cong K_{\underbrace{t,t,\dots,t}_{(k-s) }, \underbrace{t+1,t+1,\dots,t+1}_{s}}$, where $n=kt+s$ and $s\ge 0$.\\
We now show that the Tur$\acute{\text{a}}$n graph has maximum $\mathcal{ABS}$ index among all complete $k$-partite graphs.
 \begin{theorem}\label{turan}
 Let $G$ be a complete k-partite graph on n vertices. Then $\mathcal{ABS}(G)\le \mathcal{ABS}(T(n,k))$. Equality holds if and only if $G\cong T(n,k)$.
 \end{theorem}
 \begin{proof}
 Let $\Gamma\cong K_{t_1,t_2,\dots,t_k}$ be a complete $k$-partite graph with maximum $\mathcal{ABS}$ index in the class of complete $k$-partite graphs. Suppose $|t_{i}-t_{j}|\ge 2$ for some $1\le i< j\le k$. Then without loss of generality, we can  assume that $t_{1}-t_{2}\ge2$. Let $\Gamma^{\prime}=K_{t_1-1,t_2+1,\dots,t_k}$. By direct calculation, 
\begin{align*}
\mathcal{ABS}(\Gamma)-\mathcal{ABS}(\Gamma^{\prime})&=
t_1t_2\sqrt{1-\dfrac{2}{2n-(t_1+t_2)}}-(t_1-1)(t_2+1)\sqrt{1-\dfrac{2}{2n-(t_1+t_2)}}\\[2mm]&+\sum\limits_{i=3}^k t_1t_i\left(\sqrt{1-\dfrac{2}{2n-t_1-t_i}}-\sqrt{1-\dfrac{2}{2n-t_i-t_1+1}}\right)\\[2mm]&+\sum\limits_{i=3}^k t_2t_i\left(\sqrt{1-\dfrac{2}{2n-t_2-t_i}}-\sqrt{1-\dfrac{2}{2n-t_2-t_i-1}}\right)\\[2mm]
&+\sum_{i=3}^{k}t_{i}\left(\sqrt{1-\dfrac{2}{2n-t_{i}-t_{1}+1}}-\sqrt{1-\dfrac{2}{2n-t_{i}-t_{2}-1}}\right).
 \end{align*}
Since $t_{1}\ge t_{2}+2$, \begin{align}\label{peeq1}\mathcal{ABS}(\Gamma)-\mathcal{ABS}(\Gamma^{\prime})&< \sum\limits_{i=3}^k t_2t_i\left(\sqrt{1-\dfrac{2}{2n-t_1-t_i}}-\sqrt{1-\dfrac{2}{2n-t_i-t_1+1}}\right.
\notag\\[2mm]&\left.+\sqrt{1-\dfrac{2}{2n-t_2-t_i}}-\sqrt{1-\dfrac{2}{2n-t_2-t_i-1}}\right).\end{align}
 Consider $\zeta(x)=\sqrt{1-\dfrac{2}{x-t_{i}}}-\sqrt{1-\dfrac{2}{x-t_{i}+1}}$. Employing the first derivative test, we see that the  function $\zeta(x)$ is increasing. So, $\zeta(2n-t_{2}-1)\ge \zeta(2n-t_{1})$ as $t_{1}\ge t_{2}+2$, and thus from inequality (\ref{peeq1}), $\mathcal{ABS}(\Gamma)<\mathcal{ABS}(\Gamma^{\prime})$.   
 This is a contradiction. Hence $\Gamma\cong T(n,k)$.
  \end{proof}
 The join of two graphs $G_{1}$ and $G_{2}$ is obtained from $G_{1}$ and $G_{2}$ by joining each vertex of $G_{1}$ with every vertex in $G_{2}$. It is denoted by $G_{1}\vee G_{2}$. We need the following theorem to prove our next theorem.
\begin{theorem}\rm\cite{gao2018graphs,ali2024extremal}\label{kpartiteness th}
If $G$ is a graph possessing the maximum $\mathcal{ABS}$ index in the class of  connected graphs of order $n$ with vertex $k$-partiteness at most $r$ ($1\le r\le n-k$). Then there exist $k$ non-negative integers $t_1,t_2,\dots,t_k$ such that $\sum\limits_{i=1}^k t_i=n-r$ and $G\cong K_r\vee K_{t_1,t_2,\dots,t_k}$.
\end{theorem}

In the following theorem, we identify the graph with maximum $\mathcal{ABS}$ index in the class of connected $n$-order graphs that have fixed vertex $k$-partiteness $v_{k}(G)$. 
\begin{theorem}\label{KPAR}
Let $G$ be a graph of order $n$ with vertex $k$-partiteness $v_{k}(G)=r$ ($1\le r\le n-k$). Then $\mathcal{ABS}(G)\le \mathcal{ABS}(K_{r}\vee T(n-r,k))$. Equality holds if and only if $G\cong K_{r}\vee T(n-r,k)$.
\end{theorem}
\begin{proof}
Let $\Gamma$ be a graph with maximum $\mathcal{ABS}$ index in the class of graphs with $k$-partiteness $v_{k}(G)=r$.  
By Theorem~\ref{kpartiteness th}, $G\cong K_r\vee K_{t_1,t_2,\dots,t_k}$. Suppose $|t_{i}-t_{j}|\ge 2$ for some $1\le i\le j\le k$. Then without loss of generality, we can  assume that $t_{1}-t_{2}\ge2$. Let $\Gamma^{\prime}=K_{r}\vee K_{t_1-1,t_2+1,\dots,t_k}$. By direct calculation,
 $\mathcal{ABS}(\Gamma)-\mathcal{ABS}(\Gamma^{\prime})=$ \begin{align*}&~rt_1\sqrt{1-\dfrac{2}{2n-(t_1+1)}}+rt_2\sqrt{1-\dfrac{2}{2n-(t_2+1)}}+\sum\limits_{i=2}^k t_1t_i\sqrt{1-\dfrac{2}{2n-(t_1+t_i)}}\\[2mm]&+\sum\limits_{i=3}^k t_2t_i\sqrt{1-\dfrac{2}{2n-(t_2+t_i)}}-r(t_1-1)\sqrt{1-\dfrac{2}{2n-t_1}}-r(t_2+1)\sqrt{1-\dfrac{2}{2(n-1)-t_2}}\\[2mm]&-(t_1-1)(t_2+1)\sqrt{1-\dfrac{2}{2n-(t_1+t_2)}}-\sum\limits_{i=3}^k (t_1-1)t_i\sqrt{1-\dfrac{2}{2n-(t_1+t_i)+1}}\\[2mm]&-\sum\limits_{i=3}^k (t_2+1)t_i\sqrt{1-\dfrac{2}{2n-(t_2+t_i)-1}}.
\end{align*}
Since $t_{1}\ge t_{2}+2$, $\mathcal{ABS}(\Gamma)-\mathcal{ABS}(\Gamma^{\prime})<$ {\small\begin{align*}\label{peeq2} &\sum\limits_{i=3}^k t_2t_i\left(\sqrt{1-\dfrac{2}{2n-t_1-t_i}}-\sqrt{1-\dfrac{2}{2n-t_i-t_1+1}}+\sqrt{1-\dfrac{2}{2n-t_2-t_i}}-\sqrt{1-\dfrac{2}{2n-t_2-t_i-1}}\right)\notag\\[2mm]&-rt_{2}\left(\sqrt{1-\dfrac{2}{2n-t_1}}-\sqrt{1-\dfrac{2}{2n-t_1-1}}-\sqrt{1-\dfrac{2}{2n-1-t_2}}+\sqrt{1-\dfrac{2}{2n-t_2-2}}\right)\notag\\[2mm]
&=\sum\limits_{i=3}^k\big(t_2t_i(\zeta(2n-t_{1})-\zeta(2n-t_{2}-1))\big)-rt_{2}\left(\zeta_{1}(t_{1})-\zeta_{1}(t_{2}+1)\right),	
\end{align*}}
where $\zeta(x)$ is as defined in Theorem \ref{KPAR} and $\zeta_{1}(x)=\sqrt{1-\dfrac{2}{2n-x}}-\sqrt{1-\dfrac{2}{2n-x-1}}$. Since $\zeta$ and $\zeta_{1}$ are increasing functions, and as $t_{1}\ge t_{2}+2$ , $\zeta(2n-t_{2}-1)\ge \zeta(2n-t_{1})$ and $\zeta_{1}(t_{1})\ge \zeta_{1}(t_{2}+1)$. Thus, $\mathcal{ABS}(\Gamma)<\mathcal{ABS}(\Gamma^{\prime})$, a contradiction. This completes the proof of the theorem. 
\end{proof}
 
\section{$\mathcal{ABS}$ index of bipartite graphs with given vertex connectivity}\label{given3}
In this section, we identify the maximum value of the $\mathcal{ABS}$ index among the class of connected bipartite graphs with given vertex connectivity. The following graph class is crucial to this section.\\[2mm]
Let $n_{i}$ be a non-negative integer for $i=1,2,3,4,5,6$.  The graph $\overline{\mathcal{K}}[n_{1},n_{2},n_{3},n_{4},n_{5},n_{6}]$ is obtained from the graphs $H_{1}=\overline{K}_{n_{1}}, H_{2}=\overline{K}_{n_{2}}, H_{3}=\overline{K}_{n_{3}}, H_{4}=\overline{K}_{n_{4}}, H_{5}=\overline{K}_{n_{5}}$, and $H_{6}=\overline{K}_{n_{6}}$ by joining each vertex of $H_{1}$ (respectively $H_{3}$) with every vertex in $H_{4}$ and $H_{5}$ (respectively $H_{5}$ and $H_{6}$), and also joining each of vertex of $H_{2}$ with every vertex in $H_{4}$, $H_{5}$, and $H_{6}$. See Fig. \ref{F1ABS}.\\
The class of connected bipartite graphs of order $n$ and vertex connectivity $\kappa$ is denoted as $\mathcal{BG}_{n,\kappa}$. Also, the subgraph of $G$ induced by a vertex subset $S$ of $G$ is denoted as $G[S]$. The following theorem provides key information about the structure of  graphs in $\mathcal{BG}_{n,k}$ that attain the maximum $\mathcal{ABS}$ index.

\begin{theorem}\rm\cite{chen2019extremal,ali2024extremal}\label{bivertexc}
	 Let $\Gamma$ be a graph with maximum $\mathcal{ABS}$ index in the class $\mathcal{BG}_{n,\kappa}$.  Let $S$ be a minimum vertex cut set of $G$. If $G-S$ has a non-trivial component, then $G-S$ has exactly two components, namely $H_1$ and $H_2$, such that the graphs $G[S\cup V(H_1)]$ and $G[S\cup V(H_2)]$ are complete bipartite.
\end{theorem}
\begin{remark}\label{fremark}
From Theorem \ref{bivertexc}, a graph that attains the maximum $\mathcal{ABS}$ index in $\mathcal{BG}_{n,\kappa}$ must be isomorphic to one of the following graphs\\[2mm] 
(i) $\overline{\mathcal{K}}[n_{1},n_{2},n_{3},n_{4},n_{5},n_{6}]$, where  $n_{i}\ge 1$ for $i=1,3,4,6$, $n_{2}+n_{5}=\kappa$, $n_{1},n_{3}\ge n_{5}$ and $n_{4},n_{6}\ge n_{2}$.\\[2mm]
(ii) $\overline{\mathcal{K}}[0,n_{2},n_{3},n_{4},0,n_{6}]$, where $n_{3},n_{4}\ge1$,  $n_{2}=\kappa$ and $n_{6}\ge n_{2}$.\\[2mm]   
(iii) $\overline{\mathcal{K}}[n_{1},0,n_{3},0,n_{5},n_{6}]$, where $n_{1},n_{6}\ge 1$, $n_{5}=\kappa$ and $n_{3}\ge n_{5}$.   
\end{remark}
\begin{figure}[H]
\includegraphics[width=1\textwidth]{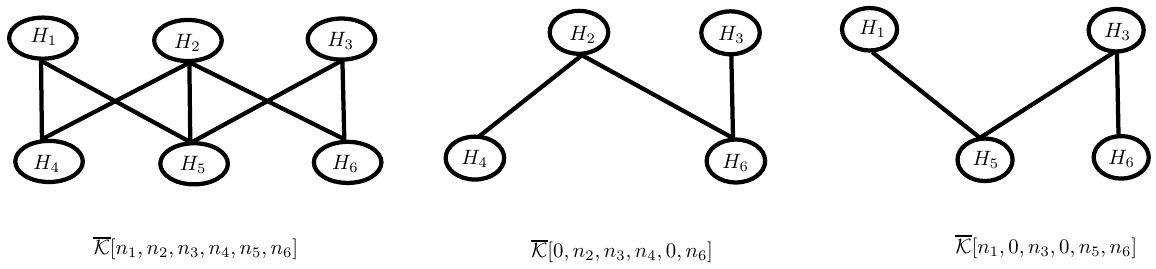}
\caption{The graphs discussed in Remark~\ref{fremark}.}
\label{F1ABS}
\end{figure}
\begin{lemma}\label{fl1}
For the graph $\overline{\mathcal{K}}[n_{1},n_{2},n_{3},n_{4},n_{5},n_{6}]$, where  $n_{i}\ge 1$ for $i=1,4,6$, $n_{3}\ge 2$, $n_{2}+n_{5}=\kappa$, $n_{1},n_{3}\ge n_{5}$ and $n_{4},n_{6}\ge n_{2}$, \[ \overline{\mathcal{K}}[n_{1},n_{2},n_{3},n_{4},n_{5},n_{6}]<\overline{\mathcal{K}}[n_{1}+n_{2}+n_{3}-1,0,1,n_{4}+n_{6}-n_{2},n_{5}+n_{2},0]\]	
\end{lemma}
\begin{proof}
Let $n=n_{1}+n_{2}+n_{3}+n_{4}+n_{5}+n_{6}$. Then
\begin{align}\label{beq1} \overline{\mathcal{K}}[n_{1},n_{2},n_{3},n_{4},n_{5},n_{6}]=n_{1}n_{4}\sqrt{1-\dfrac{2}{n-n_{3}-n_{6}}}+n_{1}n_{5}\sqrt{1-\dfrac{2}{n-n_{6}}}+n_{2}n_{4}\sqrt{1-\dfrac{2}{n-n_{3}}}\notag\\[3mm]+n_{2}n_{5}\sqrt{1-\dfrac{2}{n}}+n_{2}n_{6}\sqrt{1-\dfrac{2}{n-n_{1}}}+n_{3}n_{5}\sqrt{1-\dfrac{2}{n-n_{4}}}+n_{3}n_{6}\sqrt{1-\dfrac{2}{n-n_{1}-n_{4}}}
\end{align}
and 
\begin{align}\label{beq2} \overline{\mathcal{K}}[n_{1}+n_{2}+n_{3}-1,0,1,n_{4}+n_{6}-n_{2},n_{5}+n_{2},0]=(n_{1}+n_{2}+n_{3}-1)(n_{4}+n_{6}-n_{2})~~~~~~~\notag\\[3mm]\times\sqrt{1-\dfrac{2}{n-1}}+(n_{1}+n_{2}+n_{3}-1)(n_{5}+n_{2})\sqrt{1-\dfrac{2}{n}}+(n_{5}+n_{2})\sqrt{1-\dfrac{2}{n-n_{4}-n_{6}+n_{2}}}.
\end{align}
From equations (\ref{beq1}) and (\ref{beq2}),\\ $\overline{\mathcal{K}}[n_{1}+n_{2}+n_{3}-1,0,1,n_{4}+n_{6}-n_{2},n_{5}+n_{2},0]- \overline{\mathcal{K}}[n_{1},n_{2},n_{3},n_{4},n_{5},n_{6}]=$
\begin{align*}
&n_{1}n_{4}\left(\sqrt{1-\dfrac{2}{n-1}}-\sqrt{1-\dfrac{2}{n-n_{3}-n_{6}}}\right)+n_{1}n_{5}\left(\sqrt{1-\dfrac{2}{n}}-\sqrt{1-\dfrac{2}{n-n_{6}}}\right)\notag\\[3mm]&+n_{2}n_{4}\left(\sqrt{1-\dfrac{2}{n-1}}-\sqrt{1-\dfrac{2}{n-n_{3}}}\right)+n_{2}n_{6}\left(\sqrt{1-\dfrac{2}{n-1}}-\sqrt{1-\dfrac{2}{n-n_{1}}}\right)\notag\\[3mm]&+n^{2}_{2}\left(\sqrt{1-\dfrac{2}{n}}-\sqrt{1-\dfrac{2}{n-1}}\right)
+(n_{3}-1)n_{5}\left(\sqrt{1-\dfrac{2}{n}}-\sqrt{1-\dfrac{2}{n-n_{4}}}\right)\notag\\[3mm]&+\left[(n_{3}-1)(n_{6}-n_{2})\sqrt{1-\dfrac{2}{n-1}}+(n_{3}-1)n_{2}\sqrt{1-\dfrac{2}{n}}-(n_{3}-1)n_{6}\sqrt{1-\dfrac{2}{n-n_{1}-n_{4}}}\right]\end{align*}
\begin{align}&+(n_{3}-1)n_{4}\sqrt{1-\dfrac{2}{n-1}}+n_{1}n_{2}\left(\sqrt{1-\dfrac{2}{n}}-\sqrt{1-\dfrac{2}{n-1}}\right)+n_{1}n_{6}\sqrt{1-\dfrac{2}{n-1}}\notag\\[3mm]&-n_{6}\sqrt{1-\dfrac{2}{n-n_{1}-n_{4}}}-n_{5}\sqrt{1-\dfrac{2}{n-n_{4}}}+(n_5+n_{2})\sqrt{1-\dfrac{2}{n-n_{4}-n_{6}+n_{2}}}.\notag\end{align}
Since $\sqrt{1-\dfrac{2}{n-x}}\ge \sqrt{1-\dfrac{2}{n-y}}$ for $x\le y$, we get\\[3mm] 
$\overline{\mathcal{K}}[n_{1}+n_{2}+n_{3}-1,0,1,n_{4}+n_{6}-n_{2},n_{5}+n_{2},0]- \overline{\mathcal{K}}[n_{1},n_{2},n_{3},n_{4},n_{5},n_{6}]$
\begin{align}\label{beq3}
\ge &(n_{3}-1)n_{4}\sqrt{1-\dfrac{2}{n-1}}+n_{1}n_{6}\sqrt{1-\dfrac{2}{n-1}}-n_{6}\sqrt{1-\dfrac{2}{n-n_{1}-n_{4}}}-n_{5}\sqrt{1-\dfrac{2}{n-n_{4}}}\notag\\[3mm]&+(n_5+n_{2})\sqrt{1-\dfrac{2}{n-n_{4}-n_{6}+n_{2}}}.\end{align}
Since $n_{3}\ge n_{5}$ and $n_{1}\ge n_{5}$, from equation (\ref{beq3}),\\[2mm]  
$\overline{\mathcal{K}}[n_{1}+n_{2}+n_{3}-1,0,1,n_{4}+n_{6}-n_{2},n_{5}+n_{2},0]- \overline{\mathcal{K}}[n_{1},n_{2},n_{3},n_{4},n_{5},n_{6}]$\\[2mm]
$\ge (n_{5}-1)(n_{4}+n_{6})\sqrt{1-\dfrac{2}{n-1}}-n_{5}\sqrt{1-\dfrac{2}{n-n_{4}}}$\\[3mm]
$\ge 2(n_{5}-1)\sqrt{1-\dfrac{2}{n-1}}-n_{5}\sqrt{1-\dfrac{2}{n-n_{4}}}\ge 0$ for $n_{5}>1$. If $n_{5}=1$, then from equation (\ref{beq3}), we immediately arrive at the desired inequality because $n_{3}\ge 2$. This completes the proof of the lemma.  
\end{proof}
\begin{lemma}\label{bl2}
 $\overline{\mathcal{K}}[0,n_{2},n_{3},n_{4},0,n_{6}]< \overline{\mathcal{K}}[0,n_{2},n_{3},1,0,n_{6}+n_{4}-1]$ where $n_{3}\ge 1$, $n_{4}\ge 2$,  $n_{6}\ge n_{2}$ and $n_{2}=\kappa$.	
\end{lemma}
\begin{proof}
Let $n=n_{2}+n_{3}+n_{4}+n_{6}$. Then
\begin{align}\label{beq4}\overline{\mathcal{K}}[0,n_{2},n_{3},n_{4},0,n_{6}]&= n_{2}(n_{4}-1)\sqrt{1-\dfrac{2}{n-n_{3}}}+n_{2}\sqrt{1-\dfrac{2}{n-n_{3}}}+n_{2}n_{6}\sqrt{1-\dfrac{2}{n}}\notag\\[2mm]&+n_{3}n_{6}\sqrt{1-\dfrac{2}{n-n_{4}}}
\end{align}
 and
\begin{align}\label{beq5}\overline{\mathcal{K}}[0,n_{2},n_{3},1,0,n_{6}+n_{4}-1]&=n_{2}\sqrt{1-\dfrac{2}{n-n_{3}}}+n_{2}n_{6}\sqrt{1-\dfrac{2}{n}}+n_{2}(n_{4}-1)\sqrt{1-\dfrac{2}{n}}\notag\\[2mm]&+n_{3}n_{6}\sqrt{1-\dfrac{2}{n-1}}+n_{3}(n_{4}-1)\sqrt{1-\dfrac{2}{n-1}}.\end{align} 
From equations (\ref{beq4}) and (\ref{beq5}), and since $\sqrt{1-\dfrac{2}{n-x}}\ge \sqrt{1-\dfrac{2}{n-y}}$ for $x\le y$, we get
$\overline{\mathcal{K}}[0,n_{2},n_{3},1,0,n_{6}+n_{4}-1]-\overline{\mathcal{K}}[0,n_{2},n_{3},n_{4},0,n_{6}]\ge n_{3}(n_{4}-1)\sqrt{1-\dfrac{2}{n-1}}\ge 0.$	
\end{proof}
The following lemma can be proved using an argument similar to that of Lemma \ref{bl2}.
\begin{lemma}\label{fl3}$\overline{\mathcal{K}}[n_{1},0,n_{3},0,n_{5},n_{6}]<\overline{\mathcal{K}}[1,0,n_{3}+n_{1}-1,0,n_{5},n_{6}]$, where $n_{6}\ge 1$, $n_{1}\ge 2$ $n_{5}=\kappa$ and $n_{3}\ge n_{5}$.\end{lemma}
\begin{figure}[H]
~~~~~~~~~~~~~~~~~~~~~~~~~~~~~~\includegraphics[width=0.35\textwidth]{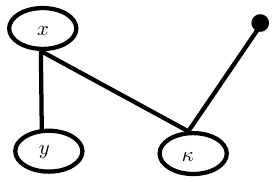}
\caption{Graph $\overline{\mathcal{K}}_{\kappa}[x,y]$}
\label{KXY}
\end{figure}
Let $\overline{\mathcal{K}}_{\kappa}[x,y]=\overline{\mathcal{K}}[x,0,1,y,\kappa,0]$ (see Fig~\ref{KXY}), where $x\ge1$, $y\ge 1$ and $x+y+1+\kappa=n$. The following theorem follow from Remark \ref{fremark} and Lemmas \ref{bl2}, \ref{fl1} and \ref{fl3}.
\begin{theorem}\label{ffffft1}
Let $G$ be a graph with maximum $\mathcal{ABS}$ index in $\mathcal{BG}_{n,k}$ and let $S$ be a minimum vertex cut set of $G$. If $G-S$ has a non-trivial component, then $G\cong K_{\kappa}[x,y]$ for some positive integers $x$ and $y$. 	
\end{theorem}
We now proceed to determine the values of $x$ and $y$ for which $\mathcal{ABS}(K_{\kappa}[x,y])$ is maximum. 
\begin{lemma}\label{fil1}
$\overline{\mathcal{K}}_{\kappa}[x+1,y-1]> 	\overline{\mathcal{K}}_{\kappa}[x,y]$, where $x+y+\kappa+1=n$ and $y-x-1+k\ge0$.
\end{lemma}
\begin{proof}
$\overline{\mathcal{K}}_{\kappa}[x,y]=xy\sqrt{1-\dfrac{2}{n-1}}+xk\sqrt{1-\dfrac{2}{n}}+k\sqrt{1-\dfrac{2}{n-y}}$\\[2mm]	
and\\[2mm] $\overline{\mathcal{K}}_{\kappa}[x+1,y-1]=(xy-x+y-1)\sqrt{1-\dfrac{2}{n-1}}+(x+1)k\sqrt{1-\dfrac{2}{n}}+k\sqrt{1-\dfrac{2}{n-y+1}}$.\\[2mm]	
Therefore,\\[2mm]
$\overline{\mathcal{K}}_{\kappa}[x+1,y-1]-\overline{\mathcal{K}}_{\kappa}[x,y]> (y-x-1+k)\sqrt{1-\dfrac{2}{n-1}}\ge 0$.
\end{proof}
\begin{lemma}\label{fil2} For $n$ even ($n\ge 8$)  and $0\le c\le \dfrac{n-2k-6}{2}$,\\ 
$$\overline{\mathcal{K}}_{\kappa}\left[\dfrac{n}{2}+c,\dfrac{n-2k-2}{2}-c\right]>\overline{\mathcal{K}}_{\kappa}\left[\dfrac{n}{2}+c+1,\dfrac{n-2k-2}{2}-c-1\right].$$
\end{lemma}
\begin{proof}
	Let $f(c)=
	\overline{\mathcal{K}}_{\kappa}\left[\dfrac{n}{2}+c,\dfrac{n-2k-2}{2}-c\right]-\overline{\mathcal{K}}_{\kappa}\left[\dfrac{n}{2}+c+1,\dfrac{n-2k-2}{2}-c-1\right].$ Then $f(c)=$ \begin{align*}2\,\sqrt {{\dfrac {n-3}{n-1}}}c+\sqrt {{\dfrac {n-3}{n-1}}}k-k\sqrt {{
				\dfrac {n-2}{n}}}+k\sqrt {{\dfrac {n+2\,k-2+2\,c}{n+2\,k+2+2\,c}}}-k
		\sqrt {{\dfrac {n+2\,k+2\,c}{n+2\,k+4+2\,c}}}+2\,\sqrt {{\dfrac {n-3}{n-1}}}.\end{align*}
	By first derivative test, it is easy to check that the function $f(c)$ is increasing, and thus $f(c)\ge f(0)$.\\[2mm]
Consider
$ f(0)=\overline{\mathcal{K}}_{\kappa}\left[\dfrac{n}{2},\dfrac{n-2k-2}{2}\right]-\overline{\mathcal{K}}_{\kappa}\left[\dfrac{n}{2}+1,\dfrac{n-2k-2}{2}-1\right]=$\begin{align}\label{ie0}k\sqrt {{\dfrac {n-3}{n-1}}} - k\sqrt {{\dfrac {n-2}{n}}}+k\sqrt {{\dfrac {
			n+2\,k-2}{n+2\,k+2}}}-k\sqrt {{\dfrac {n+2\,k}{n+2\,k+4}}}+2\,\sqrt {{
		\dfrac {n-3}{n-1}}}.\end{align}
\noindent
Claim 1:\begin{align}\label{ie1}\sqrt {{
		\dfrac {n-3}{n-1}}}>k\sqrt {{\dfrac {n+2\,k}{n+2\,k+4}}}-k\sqrt {{\dfrac {
			n+2\,k-2}{n+2\,k+2}}}.\end{align}
Squaring both sides of inequality (\ref{ie1}), we get\begin{align*}
2\,{k}^{2}\sqrt {{\dfrac {n+2\,k-2}{n+2\,k+2}}}\sqrt {{\dfrac {n+2\,k}{n
			+2\,k+4}}}>2\,{k}^{2}-{\dfrac {{4k}^{2}}{n+2\,k+2}}-{\dfrac {{4k}^{2}}{n+2\,k+4
		}}-1+\dfrac{2}{n-1}.
\end{align*}	
Now, $\left(2\,{k}^{2}\sqrt {{\dfrac {n+2\,k-2}{n+2\,k+2}}}\sqrt {{\dfrac {n+2\,k}{n
			+2\,k+4}}}\right)^2-\left(2\,{k}^{2}-{\dfrac {{4k}^{2}}{n+2\,k+2}}-{\dfrac {{4k}^{2}}{n+2\,k+4
		}}-1+\dfrac{2}{n-1}\right)^2$\\\\[2mm]
	$=\dfrac{1}{(n+2k+2)^2(n+2k+4)^2(n-1)^2}\times\\[2mm]\Big[\left( 64\,{n}^{2}-256\,n+192 \right) {k}^{6}+ \left( 128\,{n}^{3}-
	256\,{n}^{2}-640\,n+768 \right) {k}^{5}+ \left( 96\,{n}^{4}-1072\,{n}^
	{2}+352\,n+560 \right) {k}^{4}+ \left( 32\,{n}^{5}+64\,{n}^{4}-448\,{n
	}^{3}-416\,{n}^{2}+1312\,n-1056 \right) {k}^{3}+ \left( 4\,{n}^{6}+16
	\,{n}^{5}-76\,{n}^{4}-192\,{n}^{3}+632\,{n}^{2}+368\,n\right.\\\left.-2256 \right) {k
	}^{2}+ \left( -8\,{n}^{5}-24\,{n}^{4}+152\,{n}^{3}+408\,{n}^{2}-720\,n
	-1728 \right) k-{n}^{6}-6\,{n}^{5}+11\,{n}^{4}+108\,{n}^{3}+44\,{n}^{2
	}-480\,n-576\Big]> 0.$ Thus Claim 1 is true.\\[2mm]					 	
In similar lines to Claim 1, we get,
 \begin{align}\label{ie2}\sqrt {{
 		\dfrac {n-3}{n-1}}}>k\sqrt {{\dfrac {n-2}{n}}}-k\sqrt {{\dfrac {n-3}{n-1}}}.
 	\end{align}
Thus from equations (\ref{ie0}), (\ref{ie1}) and (\ref{ie2}), $f(0)>0$. Hence $f(c)>0$. Proving the inequality.\\[2mm]
\end{proof}
The following lemma can be proved using an argument similar to that of Lemma \ref{fil2}.
\begin{lemma}\label{fil3} For $n$ odd $(n\ge 7)$ and $0\le c\le \dfrac{n-2k-5}{2}$, 
	$$\overline{\mathcal{K}}_{\kappa}\left[\dfrac{n-1}{2}+c,\dfrac{n-2k-1}{2}-c\right]>\overline{\mathcal{K}}_{\kappa}\left[\dfrac{n-1}{2}+c+1,\dfrac{n-2k-1}{2}-c-1\right].$$
\end{lemma}
\begin{theorem}\label{fffffft2}
(i) If $n$ is even, then $\mathcal{K}_{\kappa}[x,y]\le \mathcal{K}_{\kappa}\left[\dfrac{n}{2},\dfrac{n-2k-2}{2}\right]$.	
Equality holds if and only if $x=\dfrac{n}{2}$ and $y=\dfrac{n-2k-2}{2}$.\\[2mm]
(ii) If $n$ is odd, then $\mathcal{K}_{\kappa}[x,y]\le \mathcal{K}_{\kappa}\left[\dfrac{n-1}{2},\dfrac{n-2k-1}{2}\right]$.	
Equality holds if and only if $x=\dfrac{n-1}{2}$ and $y=\dfrac{n-2k-1}{2}$.
\end{theorem}
\begin{proof}
By Lemmas \ref{fil1} and \ref{fil2},
\begin{align}\label{fff1}\mathcal{K}_{\kappa}[\kappa, n-2\kappa-1]< \mathcal{K}_{\kappa}[\kappa+1, n-2\kappa-2]<\cdots< \mathcal{K}_{\kappa}\left[\dfrac{n}{2},\dfrac{n-2\kappa-2}{2}\right]\end{align}
and 
\begin{align}\label{f222}
\mathcal{K}_{\kappa}\left[\dfrac{n}{2},\dfrac{n-2\kappa-2}{2}\right]> \mathcal{K}_{\kappa}\left[\dfrac{n}{2}+1,\dfrac{n-2\kappa-2}{2}-1\right]>\cdots>\mathcal{K}_{\kappa}\left[n-\kappa-2, 1\right].	
\end{align}
Thus, from equations (\ref{fff1}) and (\ref{f222}), we get the desired result. Proof of (ii) follows from Lemmas  \ref{fil1} and \ref{fil3} in a similar manner to that of (i).
\end{proof}
\begin{remark}\label{empty}
Let $G$ be a graph with maximum $\mathcal{ABS}$ index in $\mathcal{BG}_{n.\kappa}$ and let $S$ be a minimum vertex cut set of $G$. If $G-S$ is an empty graph, then $S$ is a partition set of $G$ and $G\cong K_{\kappa,n-\kappa}$ (because addition of an edge increases $\mathcal{ABS}$ index).    	
\end{remark}
The main result of this section is presented below.
\begin{theorem}
Let $G\in \mathcal{BG}_{n,\kappa}$ with $n\ge 7$.  \begin{enumerate}[(i)]
\item For $\kappa\in\left\{\dfrac{n-2}{2},\dfrac{n-1}{2}, \dfrac{n}{2}\right\}$, $\mathcal{ABS}(G)\le \kappa(n-\kappa)\sqrt{1-\dfrac{2}{n}}$. Equality holds if and only if $G\cong K_{\kappa,n-\kappa}$.
\item For $n$ even and $1\le k\le \dfrac{n-4}{2}$,  $$\mathcal{ABS}(G)\le \dfrac{n(n-2k-2)}{4}\sqrt{1-\dfrac{2}{n-1}}+\dfrac{nk}{2}\sqrt{1-\dfrac{2}{n}}+k\sqrt{1-\dfrac{4}{2k+n+2}}.$$ Equality holds if and only if $G\cong \overline{\mathcal{K}}_{\kappa}\left[\dfrac{n}{2},\dfrac{n-2k-2}{2}\right]$.
\item For  $n$ odd and $1\le k\le \dfrac{n-3}{2}$, $$\mathcal{ABS}(G)\le \dfrac{(n-1)(n-2k-1)}{4}\sqrt{1-\dfrac{2}{n-1}}+\dfrac{(n-1)k}{2}\sqrt{1-\dfrac{2}{n}}+k\sqrt{1-\dfrac{4}{2k+n+1}}.$$ Equality holds if and only if $G\cong \overline{\mathcal{K}}_{\kappa}\left[\dfrac{n-1}{2},\dfrac{n-2k-1}{2}\right]$.
\end{enumerate}
\end{theorem}
\begin{proof}
Let $G\in \mathcal{BG}_{n,\kappa}$ with  maximum $\mathcal{ABS}$ index and let $S$ be a minimum vertex cut set of $G$. Let $\kappa\in\left\{\dfrac{n-1}{2}, \dfrac{n}{2}\right\}$. If $G-S$ has a non-trivial component, then by Theorem \ref{ffffft1},  $G\cong \overline{\mathcal{K}}_{\kappa}[x,y]$ with $x\ge 1$ and $y\ge \kappa$. So, $|G|=x+y+\kappa+1>n$, a contradiction. Hence $G-S$ is an empty graph, and so by Remark \ref{empty}, $G\cong \mathcal{K}_{\kappa,n-\kappa}$.\\[2mm]
Let $\kappa=\dfrac{n-2}{2}$. Then from Remark \ref{empty} and by Theorems \ref{ffffft1} and \ref{fffffft2}, we get, $G\cong K_{\dfrac{n-2}{2},\dfrac{n+2}{2}}$ or $\overline{\mathcal{K}}_{\kappa}\left[\dfrac{n-2}{2},1\right]$. Since $\mathcal{ABS}\left(K_{\dfrac{n-2}{2},\dfrac{n+2}{2}}\right)>\mathcal{ABS}\left(\overline{\mathcal{K}}_{\kappa}\left[\dfrac{n-2}{2},1\right]\right)$, $G\cong K_{\dfrac{n-2}{2},\dfrac{n+2}{2}}$.        
Assume that $1\le k\le \dfrac{n-3}{2}$. If $n$ is even, then from Remark \ref{empty} and by Theorems \ref{ffffft1} and \ref{fffffft2}, $G\cong K_{\kappa,n-\kappa}$  or $G\cong \overline{\mathcal{K}}_{\kappa}\left[\dfrac{n}{2},\dfrac{n-2k-2}{2}\right]$.
Since ${\small\mathcal{ABS}\left(\overline{\mathcal{K}}_{\kappa}\left[\dfrac{n}{2},\dfrac{n-2k-2}{2}\right]\right)}\\\ge 
\mathcal{ABS}\left(\overline{\mathcal{K}}_{\kappa}\left[n-k-2,1\right]\right)>\mathcal{ABS}\left(K_{\kappa,n-\kappa}\right)$, $G\cong \overline{\mathcal{K}}_{\kappa}\left[\dfrac{n}{2},\dfrac{n-2k-2}{2}\right]$. Otherwise $n$ is odd. Similar to the case $n$ is even, we get $G\cong \overline{\mathcal{K}}_{\kappa}\left[\dfrac{n-1}{2},\dfrac{n-2k-1}{2}\right]$.    
\end{proof}
\section*{Conclusion}
This work characterizes graphs with the maximum \(\mathcal{ABS}\) index in the following classes: (i) connected graphs with \(n\) vertices and \(p\) cut-vertices; (ii) connected graphs of order \(n\) with a vertex \(k\)-partiteness \(v_k(G) = r\); and (iii) connected bipartite graphs of order \(n\) with a fixed vertex connectivity \(\kappa\). This study provides solutions to some of the many open problems proposed in \cite{ali2024extremal}. Several open extremal problems related to this index require further investigation in future research.
\section*{Declarations}
\textbf{Availability of data and material}: Manuscript has no associated data.\\
\textbf{Competing interest}: The authors declare that they have no competing interest.\\
\textbf{Funding}: No funds received.\\
\textbf{Authors' contributions:} All authors have equally contributed.

\bibliography{name}
\bibliographystyle{abbrv}
\end{document}